\documentclass[11pt]{amsart}
\usepackage{amssymb}
\usepackage{xypic}

\newtheorem{theorem}{Theorem}[section]
\newtheorem{lemma}[theorem]{Lemma}
\newtheorem{corollary}[theorem]{Corollary}
\newtheorem{proposition}[theorem]{Proposition}
\theoremstyle{remark}
\newtheorem{remark}[theorem]{Remark}
\theoremstyle{remark}

\newcommand\Md{{\text{\rm Mod}}}

\newcommand\Set{{{\mathcal S}et}}
\newcommand{\Ima}{\mathop{\rm Im}\nolimits}
\newcommand{\ima}{\mathop{\rm im}\nolimits}

\newcommand\colim{\mathop{\rm colim}}

\newcommand\Eq{\mathop{\rm Eq}}

\newcommand\opp{{^{\text{\rm op}}}} % the opposite (category, ring, etc.)

\newcommand\CA{{\mathcal A}}
\newcommand\CB{{\mathcal B}}
\newcommand\CC{{\mathcal C}}

\newcommand\CE{{\mathcal E}}

\newcommand\CG{{\mathcal G}}

\newcommand\CS{{\mathcal S}}

\newcommand\id{{\bf 1}}

\begin{document}

\title{Localizations, colocalizations and non additive $*$-objects}

\author{George Ciprian Modoi}

\address{"Babe\c s-Bolyai" University\\ RO-400084, Cluj-Napoca
\\ Romania}

\thanks{The author was supported by the grant PN2CD-ID-489}

\email[George Ciprian Modoi]{cmodoi@math.ubbcluj.ro}


\begin{abstract}  Given a pair of adjoint functors between two arbitrary
categories it induces mutually inverse equivalences between the
full subcategories of the initial ones, consisting of objects for
which the arrows of adjunction are isomorphisms. We investigate
some cases in which these subcategories may be better
characterized. One application is the construction of cellular
approximations. Other is the definition and the characterization
of (weak) $*$-objects in the non additive case.
\end{abstract}
\maketitle

\section*{Introduction}

In mathematics the concept of localization has a long history. The
origin of the concept is the study of some properties of maps
around a point of a topological space. In the algebraic sense, the
localization provides a method to invert some morphisms in a
category. Making abstraction of some technical set theoretic
problems, given a class of morphisms $\Sigma$ in a category $\CA$,
there is a category $\CA[\Sigma^{-1}]$ and a functor
$\CA\to\CA[\Sigma^{-1}]$ universal with the property that it sends
any morphism in $\Sigma$ to an isomorphism. This functor will be
called a {\em localization}, if it has a right adjoint, which will
be frequently fully faithful. Dually this functor is called a {\em
colocalization} provided that it has a left adjoint.

One of the starting point of this paper is the observation that
the consequences of the duality between localization and
colocalization were not exhausted. For example the concept,
borrowed from topology, of cellular approximation in arbitrary
category is a particular case of a colocalization, fact remarked
for example in \cite{FGS}. Some results concerning the cellular
approximation may be deduced in a formal, categorical way by
stressing this duality. On the other hand the same formal
techniques are useful in the study of so called $*$-modules,
defined as in \cite{CF}.

Now let us present the organization and the main results of the
paper. In the first section we set the notations, we define the
main notions used throughout of the paper and we record some easy
properties concerning these notions.

In Section \ref{eiaf} are stated the formal results, on which it
is based the rest of the paper. There are three main results here:
First Theorem \ref{rladj} where are given necessary and sufficient
conditions for a pair of adjoint functors to induce an equivalence
between the full subcategories consisting of colocal respectively
local objects with respect to these functors (for the definition
of a (co)local object see Section \ref{not}). Second and third
Proposition \ref{wstarob} and Theorem \ref{starob} which represent
the formal characterization of a non additive (weak) $*$-object.

Section \ref{repeq} contains a non additive version of a theorem
of Menini and Orsatti in \cite{MO}. Consider an object (or a set
of objects) $A$ of a category $\CA$, and the category of all
contravariant functors $[\CE\opp,\Set]$ where $\CE$ is the full
subcategory of $\CA$ containing the object(s) $A$, situation which
is less general but more comprehensive that the hypotheses of
Section \ref{repeq}. Under appropriate assumptions, mutually
inverse equivalences between two full subcategories of $\CA$ and
$[\CE\opp,\Set]$ are represented by $A$ in the sense that they are
realized by restrictions of the representable functor
$H_A=\CA(A,-)$ and of its left adjoint (see Theorem \ref{repres}).

Provided that $\CA$ is a cocomplete, well copowered, balanced
category with epimorphic images, and $A$ is a set of objects of
$\CA$, it is shown in Section \ref{ecc} that the inclusion of the
subcategory of $H_A$-colocal objects has a right adjoint (see
Theorem \ref{ctoa}). Consequently fixing an object $A$ in such a
category, every object $X$ will have an $A$-cellular
approximation.

In Section \ref{nonadd} we define and characterize the notions of
a (weak) $*$-act over a monoid, in Proposition \ref{wstaract} and
Theorem \ref{staract}, providing in this way a translation of the
notion of (weak) $*$-module in these new settings. It is
interesting to note that our approach may be continued by
developing a theory analogous with so called tilting theory for
modules. The Morita theory for the category of acts over monoids
is a consequence of our results.

\section{Notations and preliminaries}\label{not}

All subcategories which we consider are full and closed under
isomorphisms, so if we speak about a class of objects in a
category we understand also the respective subcategory. For a
category $\CA$ we denote by $\CA^\to$ the category of all
morphisms in $\CA$. We denote by $\CA(-,-)$ the bifunctor
assigning to any two objects of $\CA$ the set of all morphisms
between them.

Consider a functor $H:\CA\to\CB$. The (essential) image of $H$ is
the subcategory $\Ima H$ of $\CB$ consisting of all objects
$Y\in\CB$ satisfying $Y\cong H(X)$ for some $X\in\CA$. In contrast
we shall denote by $\ima\alpha$ the categorical notion of image of
a morphism $\alpha\in\CA^\to$. A morphism $\alpha\in\CA^\to$ is
called an {\em $H$-equivalence}, provided that $H(\alpha)$ is an
isomorphism. We denote by $\Eq(H)$ the subcategory of $\CA^\to$
consisting of all $H$-equivalences. An object $X\in\CA$ is called
{\em $H$-local} ({\em $H$-colocal}) if, for any $H$-equivalence
$\epsilon$, the induced map $\epsilon_*=\CA(\epsilon,X)$
(respectively, $\epsilon^*=\CA(X,\epsilon)$) is bijective, that
means it is an isomorphism in the category $\Set$ of all sets. We
denote by $\CC^H$ and $\CC_H$ the full subcategories of $\CA$
consisting of all $H$-local, respectively $H$-colocal objects. For
objects $X',X\in\CA$, we say that $X'$ is a {\em retract} of $X$
if there are maps $\alpha:X'\to X$ and $\beta:X\to X'$ in $\CA$
such that $\beta\alpha=1_{X'}$. We record without proof the
following properties relative to the above considered notions:

\begin{lemma}\label{2out3} The following hold:
\begin{itemize}
 \item[{\rm a)}] $\Eq(H)$ is closed under
retracts in $\CA^\to$. \item[{\rm b)}] $\Eq(H)$ satisfies the `two
out of three' property, namely if $\alpha,\beta\in\CA^\to$ are
composable morphisms, then if two of the morphisms
$\alpha,\beta,\beta\alpha$ are $H$-equivalences, then so is the
third. \item[{\rm c)}] The subcategory $\CC^H$ (respectively
$\CC_H$) is closed under limits (respectively colimits) and
retracts in $\CA$.
\end{itemize}
\end{lemma}

Moreover if every object of $\CA$ has a left (right) approximation
with an $H$-local (colocal) object, in a sense becoming precise in
the hypothesis of the Lemma bellow, then we are in the situation
of a localization (colocalization) functor, as it may be seen
from:

\begin{lemma}\label{fractions}
If for every $X\in\CA$ there is an $H$-equivalence $X\to X^H$ with
$X^H\in\CC^H$ (respectively, $X_H\to X$ with $X_H\in\CC_H$), then
the assignment $X\mapsto X^H$ ($X\mapsto X_H$) is functorial and
defines a left (right) adjoint of the inclusion functor
$\CC^H\to\CA$ ($\CC_H\to\CA$). Moreover the left (right) adjoint
of the inclusion functor sends every map $\alpha\in\Eq(H)$ into an
isomorphism and it is universal relative to this property.
\end{lemma}

\begin{proof} Straightforward. (The first statement was also noticed in \cite[1.6]{DL}).
\end{proof}

In the sequel we consider a pair of adjoint functors $H:\CA\to\CB$
at the right and $T:\CB\to\CA$ at the left, where $\CA$ and $\CB$
are arbitrary categories. We shall symbolize this situation by
$T\dashv H$. Consider also the arrows of adjunction
\[\delta:T\circ H\to\id_\CA\hbox{ and }\eta:\id_\CB\to H\circ T.\]
Note that, for all $X\in\CA$ and all $Y\in\CB$ we obtain the
commutative diagrams in $\CB$ and $\CA$ respectively:
\begin{equation}\label{diaadj}\diagram H(X)\rto^{\eta_{H(X)}\hskip20pt}\drdouble_{1_{H(X)}}&(H\circ T\circ H)(X)\dto^{H(\delta_X)}\\
&H(X)\enddiagram\hbox{ and }
\diagram T(Y)\rto^{T(\eta_Y)\hskip20pt}\drdouble_{1_{T(Y)}}&(T\circ H\circ T)(Y)\dto^{\delta_{T(Y)}}\\
&T(Y)\enddiagram\end{equation} showing that $H(X)$ and $T(Y)$ are
retracts of $(H\circ T\circ H)(X)$, respectively $(T\circ H\circ
T)(Y)$. Corresponding to the adjoint pair considered above, we
define the following full subcategories of $\CA$ and $\CB$:
\[\CS_H=\{X\in\CA\mid\delta_X:(T\circ H)(X)\to X\hbox{ is an
isomorphism}\},\]  and respectively
\[\CS^T=\{Y\in\CB\mid\eta_Y:Y\to(H\circ T)(Y)\hbox{ is an
isomorphism}\}.\] The objects in $\CS_H$ and $\CS^T$ are called
{\em $\delta$-reflexive}, respectively {\em $\eta$-reflexive}.
Note that $H$ and $T$ restrict to mutually inverse equivalences of
categories between $\CS_H$ and $\CS^T$ and these subcategories are
the largest of $\CA$ and $\CB$ respectively, enjoying this
property.

\section{An equivalence induced by adjoint functors}\label{eiaf}

In this section we fix a pair of adjoint functors $T\dashv H$
between two arbitrary categories  $\CA$ and $\CB$, as in Section
\ref{not}.

\begin{lemma}\label{inclusions} The following inclusions hold:
\begin{itemize}
\item[{\rm a)}] $\CS_H\subseteq\Ima T\subseteq\CC_H\subseteq\CA$.
\item[{\rm b)}] $\CS^T\subseteq\Ima H\subseteq\CC^T\subseteq\CB$.
\end{itemize}
\end{lemma}

\begin{proof}
a) The first inclusion is obvious. For the second inclusion
observe that for all $\epsilon\in\Eq(H)$ and all $Y\in\CB$ the
isomorphism in $\Set^\to$
\[\epsilon^*=\CA(T(Y),\epsilon)\cong\CB(Y,H(\epsilon))\] shows
that $\epsilon^*$ is bijective. The inclusions from b) follow by
duality.
\end{proof}

\begin{lemma}\label{smallest}
Let $\CC$ be a subcategory of $\CA$ such that the inclusion
functor $I:\CC\to\CA$ has a right adjoint $R:\CA\to\CC$ and the
arrow of the adjunction $\mu_X:(I\circ R)(X)\to X$ is an
$H$-equivalence for all $X\in\CA$.  Then $\mu_X$ is an isomorphism
for all $X\in\CC_H$, and consequently $\CC_H\subseteq\CC$.
\end{lemma}

\begin{proof} Let $X\in\CC_H$. Since $\mu_X\in\Eq(H)$, we deduce
that the induced map
\[\mu_X^*:\CA(X,(I\circ R)(X))\to\CA(X,X)\] is bijective,
consequently there is a morphism $\mu_X':X\to(I\circ R)(X)$ such
that $\mu_X\mu_X'=1_X$. Since $R\circ I\cong\id_\CC$ naturally,
and $\mu$ is also natural, we obtain a commutative diagram
\[\diagram
(I\circ R)(X)\rrto^{(I\circ R)(\mu_X')\hskip15pt}\dto_{\mu_X}\drrdouble&&(I\circ R\circ I\circ R)(X)\dto^{\mu_{(I\circ R)(X)}}\\
X\rrto_{\mu_X'}&&(I\circ R)(X)
\enddiagram\]
showing that $\mu_X'\mu_X=1_{(I\circ R)(X)}$, hence $\mu_X$ is an
isomorphism.
\end{proof}

\begin{lemma}\label{sisimh} If $\CS^T=\Ima H$ then $\CS_H=\CC_H$.
Dually if $\CS_H=\Ima T$ then $\CS^T=\CC^T$.
\end{lemma}

\begin{proof} Consider an arbitrary object $X\in\CA$. By
hypothesis $H(X)\in\CS^T$, so $\eta_{H(X)}$ is an isomorphism.
Together with diagrams (\ref{diaadj}), this implies $H(\delta_X)$
is an isomorphism. Thus $\delta_X:(T\circ H)(X)\to X$ is an
$H$-equivalence, and we know $(T\circ H)(X)\in\Ima
T\subseteq\CC_H$. As we learned from Lemma \ref{fractions}, this
means that the assignment $X\mapsto(T\circ H)(X)$ defines a right
adjoint of the inclusion of $\CC_H$ in $\CA$. If in addition
$X\in\CC_H$, then $\delta_X$ is an isomorphism by Lemma
\ref{smallest}, proving the inclusion $\CC_H\subseteq\CS_H$. Since
the converse inclusion is always true, the conclusion holds.
\end{proof}

\begin{remark}\label{rladj}
From the proof of Lemma \ref{sisimh} we can see that the condition
$\CS^T=\Ima H$ implies that the functor $\CA\to\CC_H$,
$X\mapsto(T\circ H)(X)$ is the right adjoint of the inclusion
functor of $\CS_H=\CC_H$ into $\CA$. Dually if $\CS_H=\Ima T$,
then the functor $\CB\to\CC^T$, $Y\mapsto(H\circ T)(Y)$ is the
left adjoint of the inclusion functor of $\CS^T=\CC^T$ into $\CB$.
\end{remark}

\begin{theorem}\label{sisc} The following are equivalent:
\begin{itemize}
\item[{\rm (i)}] $\CS^T=\CC^T$. \item[{\rm (ii)}] $\CS^T=\Ima H$.
\item[{\rm (iii)}] $\CS_H=\CC_H$. \item[{\rm (iv)}] $\CS_H=\Ima
T$. \item[{\rm (v)}] The functors $H$ and $T$ induce mutually
inverse equivalence of categories between $\CC_H$ and $\CC^T$.
\end{itemize}
\end{theorem}

\begin{proof} The equivalence of the conditions (i)--(iv) follows by Lemmas
\ref{inclusions} and \ref{sisimh}. Finally the equivalent
conditions (i)--(ii) are also equivalent to (v), because $\CS_H$
and $\CS^T$ are the largest subcategories of $\CA$ and $\CB$ for
which $H$ and $T$ restrict to mutually inverse equivalences.
\end{proof}

\begin{corollary}\label{tisff}
The adjoint functors $T\dashv H$ induce mutually inverse
equivalences $\CC_H\rightleftarrows\CB$ if and only if $T$ is
fully faithful. Dually the adjoint pair induces equivalences
$\CA\rightleftarrows\CC^T$ if and only if $H$ is fully faithful.
\end{corollary}

\begin{proof} The functor $T$ is fully faithful exactly if the unit of the
adjunction $\eta:\id_\CB\to(H\circ T)$ is an isomorphism, or
equivalently, $\CS^T=\CB$. Now, Theorem \ref{sisc} applies.
\end{proof}

Theorem \ref{sisc} and Corollary \ref{tisff} generalize
\cite[Theorem 1.6 and Corollary 1.7]{CGW}, where the work is done
in the setting of abelian categories, and the proof stresses the
abelian structure. These results may be also compared with
\cite[Theorem 1.18]{CM}, where the framework is also that of
abelian categories.

We consider next other two subcategories of $\CA$ and $\CB$
respectively:
\[\CG_H=\{X\in\CA\mid\delta_X:(T\circ H)(X)\to X\hbox{ is an epimorphism}\},\]
\[\CG^T=\{Y\in\CB\mid\eta_Y:Y\to(H\circ T)(Y)\hbox{ is a monomorphism}\}.\]
The dual character of all considerations in the present Section
continues to hold for $\CG_H$ and $\CG^T$.

\begin{lemma}\label{gclosepi} The following statements hold:
\begin{itemize} \item[{\rm a)}] The subcategory $\CG_H$ (respectively $\CG^T$) is closed
under quotient objects (subobjects). \item[{\rm b)}] $\Ima
T\subseteq\CG_H$ (respectively $\Ima H\subseteq\CG^T$).
\end{itemize}
\end{lemma}

\begin{proof} a) Let $\alpha:X'\to X$ be an epimorphism in $\CA$ with
$X'\in\CG_H$. Since $\delta$ is natural, we obtain the equality
$\alpha\delta_{X'}=\delta_X(T\circ H)(\alpha)$, showing that
$\delta_X$ is an epimorphism together with $\alpha\delta_{X'}$.

b) From the diagrams (\ref{diaadj}), we see that $\delta_{T(Y)}$
is right invertible, so it is an epimorphism for any $Y\in\CB$.
Thus $\Ima T\subseteq\CG_H$.
\end{proof}

The subcategory $\CG_H$ of $\CA$ is more interesting in the case
when $\CA$ has {\em epimorphic images}, what means that it has
images and the factorization of a morphism through its image is a
composition of an epimorphism followed by a monomorphism (for
example, $\CA$ has epimorphic images, provided that it has
equalizators and images, by \cite[Chapter 1, Proposition
10.1]{BM}). Suppose also that $\CA$ is {\em balanced}, that is
every morphism which is both epimorphism and monomorphism is an
isomorphism. Thus every factorization of a morphism as a
composition of an epimorphism followed by a monomorphism is a
factorization through image, by \cite[Chapter 1, Proposition
10.2]{BM}. With these hypotheses it is not hard to see that the
factorization of a morphism through its image is functorial, that
means the assignment $\alpha\mapsto\ima\alpha$ defines a functor
$\CA^\to\to\CA$.

\begin{proposition}\label{imradj}
If $\CA$ is a balanced category with epimorphic images, then the
functor $\CA\to\CG_H$, $X\mapsto\ima\delta_X$ is a right adjoint
of the inclusion functor $\CG_H\to\CA$.
\end{proposition}

\begin{proof} By hypothesis $\ima\delta_{X}$ is a
quotient of $(H\circ T)(Y)$ and $H(T(Y))\in\CG_H$, so the functor
$\CA\to\CG_H$, $X\mapsto\ima\delta_X$ is well defined, by Lemma
\ref{gclosepi}. Let now $\alpha:X'\to X$ in $\CA^\to$, where
$X'\in\CG_H$ and $X\in\CA$. Since $\delta_{X'}$ is an epimorphism,
it follows
\[\ima\alpha=\ima\left(\alpha\delta_{X'}\right)=\ima\left(\delta_X(T\circ
H)(\alpha)\right)\subseteq\ima\delta_X,\] so $\alpha$ factors
through $\ima\delta_X$. This means that the map
\[\CA(X',\ima\delta_X)\to\CA(X',X)\] is surjective. But it is also
injective since the functor $\CA(X',-)$ preserves monomorphisms,
and the conclusion follows.
\end{proof}

\begin{corollary}\label{csubsetg}
If $\CA$ is a balanced category with epimorphic images, then the
morphism $\ima\delta_X\to X$ is an $H$-equivalence and
$\CC_H\subseteq\CG_H$.
\end{corollary}

\begin{proof} The second statement of the conclusion follows from the first one
by using  Proposition \ref{imradj} and Lemma \ref{smallest}. But
$H$ carries the monomorphism $\ima\delta_X\to X$ into a
monomorphism in $\CB$, because $H$ is a right adjoint. Moreover,
since $H(\delta_X)$ is right invertible, the same is true for the
morphism $H(\ima\delta_X)\to H(X)$, as we may see from the
commutative diagram
\[\diagram (H\circ T\circ
H)(X)\rrto^{\hskip5pt H(\delta_X)}\drto&&H(X)\\
&H(\ima\delta_X)\urto&\enddiagram.\]
\end{proof}

\begin{proposition}\label{wstarob}
Suppose both $\CA$ and $\CB$ are balanced categories with
epimorphic images. The following are equivalent:
\begin{itemize}
\item[{\rm (i)}] The pair of adjoint functors $T\dashv H$ induces
mutually inverse equivalences $\CC_H\rightleftarrows\CG^T$.
\item[{\rm (ii)}] $\eta_Y:Y\to(H\circ T)(Y)$ is an epimorphism for
all $Y\in\CB$.
\end{itemize}
\end{proposition}

\begin{proof} (i)$\Rightarrow$(ii). Denote $Y'=\ima\eta_Y$. Then the unit
$\eta_Y$ of adjunction factors as $Y\to Y'\to(H\circ T)(Y)$, where
the epimorphism $Y\to Y'$ is a $T$-equivalence by the dual of
Corollary \ref{csubsetg}, and $Y'\to(H\circ T)(Y)$ a monomorphism.
Since $(H\circ T)(Y)\in\Ima H\subseteq\CG^T$ and $\CG^T$ is closed
under subobjects, we deduce $Y'\in\CG^T$. Now (i) implies that
$\eta_{Y'}$ is an isomorphism, so the diagram
\[\diagram Y\rto\dto_{\eta_Y}&Y'\dto^{\eta_{Y'}}\\
    (H\circ T)(Y)\rto^{\cong}&(H\circ T)(Y')\enddiagram\]
proves (ii).

(ii)$\Rightarrow$(i). Condition (ii) implies that $T(\eta_Y)$ is
an epimorphism, for every $Y\in\CB$, since $T$ preserves
epimorphisms. But it is also left invertible by diagrams
\ref{diaadj}. Thus it is invertible, with the inverse
$\delta_{T(Y)}$. We have just shown that $\CS_H=\Ima T$, hence
Theorem \ref{sisc} tells us that $\CC_H$ and $\CC^T$ are
equivalent via $H$ and $T$. Finally, since $\CB$ is balanced,
clearly $\CG^T=\CS^T=\CC^T$.
\end{proof}

Combining Proposition \ref{wstarob} and its dual we obtain:

\begin{theorem}\label{starob}
Suppose both $\CA$ and $\CB$ are balanced categories with
epimorphic images. The following are equivalent:
\begin{itemize}
\item[{\rm (i)}] The pair of adjoint functors $T\dashv H$ induces
mutually inverse equivalences $\CG_H\rightleftarrows\CG^T$.
\item[{\rm (ii)}] $\delta_X:(T\circ H)(X)\to X$ is a monomorphism
for all $X\in\CA$ and $\eta_Y:Y\to(H\circ T)(Y)$ is an epimorphism
for all $Y\in\CB$.
\end{itemize}
\end{theorem}

Remark that \cite[Proposition 2.2.4 and Theorem 2.3.8]{CF} provide
characterizations of (weak) $*$-modules which are analogous to
Proposition \ref{wstarob} and Theorem \ref{starob} above. These
results will be used in Section \ref{nonadd}, for defining the
corresponding notions in a non additive situation.

\section{Representable equivalences}\label{repeq}

Overall in this section $\CA$ is a cocomplete category and $\CE$
is small category. Denote by $[\CE\opp,\Set]$ the category of all
contravariant functors from $\CE$ into $\Set$. Then we view $\CE$
as a subcategory of $[\CE\opp,\Set]$, via the Yoneda embedding
$\CE\to[\CE\opp,\Set]$, $e\mapsto\CE(-,e)$. For simplicity, we
shall write $[Y',Y]$ for $[\CE\opp,\Set](Y',Y)$, where
$Y',Y\in[\CE\opp,\Set]$. For every $Y\in[\CE\opp,\Set]$ denote by
$\CE\downarrow Y$ the comma category whose objects are of the form
$(e,y)$ with $e\in\CE$ and $y\in Y(e)$ and whose morphisms are
\[(\CE\downarrow Y)((e',y'),(e,y))=\{\alpha\in\CE(e',e)\mid Y(\alpha)(y')=y\}.\]
The projection functor $\CE\downarrow Y\to\CE$ is given by
$(e,y)\mapsto e$ and $\alpha\mapsto\alpha$ for all
$(e,y)\in\CE\downarrow Y$ and all $\alpha\in(\CE\downarrow
Y)((e',y'),(e,y))$. Observe then that the subcategory $\CE$ is
{\em dense} in $[\CE\opp,\Set]$, what means, for every
$Y\in[\CE\opp,\Set]$ it holds \[Y\cong\colim((\CE\downarrow
Y)\to\CE\to[\CE\opp,\Set])=\colim_{(e,y)\in\CE\downarrow
Y}\CE(-,e),\] where the last notation is a shorthand for the
previous colimit.

For a functor $A:\CE\to\CA$, consider the left Kan extension of
$A$ along the Yoneda embedding:
\[T_A:[\CE\opp,\Set]\to\CA, T_A(Y)=\colim_{(e,y)\in\CE\downarrow
Y}A(e),\] which may be characterized as the unique, up to a
natural isomorphism, colimit preserving functor
$[\CE\opp,\Set]\to\CA$, mapping $\CE(-,e)$ into $A(e)$ for all
$e\in\CE$.  The functor $T_A$ has a right adjoint, namely the
functor
\[H_A:\CA\to[\CE\opp,\Set], H_A(X)=\CA(A(-),X).\] In order to use the results of Section
\ref{not}, we remaind the notations made there, namely let
$\delta:T_A\circ H_A\to\id_\CA$ and $\eta:\id_\CB\to H_A\circ T_A$
be the arrows of adjunction. For simplicity we shall replace in
the next considerations the subscript $H_A$ and the superscript
$T_A$ with $A$. So objects in $\CC_A$, $\CC^A$, $\CG_A$ and
$\CG^A$ will be called {\em$A$-colocal}, {\em$A$-local},
{\em$A$-generated}, respectively {\em$A$-cogenerated}.

We consider overall in this section two subcategories
$\CC\subseteq\CA$ and $\CC'\subseteq[\CE\opp,\Set]$, such that
$\CC$ is closed under taking colimits and retracts and $\CC'$ is
closed under taking limits and retracts.

\begin{lemma}\label{repadj} Let $H:\CC\rightleftarrows\CC':T$ be a pair of adjoint
functors $T\dashv H$, and denote $A:\CE\to\CA$ the functor given
by $A(e)=T(\CE(-,e))$. If $\CE(-,e)\in\CC'$ for all $e\in\CE$ then
$H$ is naturally isomorphic to the restriction of $H_A$ and $T$ is
naturally isomorphic to the restriction of $T_A$, $\Ima
H_A\subseteq\CC'$ and $\Ima T_A\subseteq\CC$. If moreover $T$ is
fully faithful, then arrow $\delta_X:(T_A\circ H_A)(X)\to X$ is an
$H_A$-equivalence for all $X\in\CA$.
\end{lemma}

\begin{proof} Using Yoneda lemma, we
have the natural isomorphisms for every $X\in\CC$, and every
$e\in\CE$:
\[H(X)(e)\cong[\CE(-,e),H(X)]\cong\CA(T(\CE(-,e)),X)=\CA(A(e),X)=H_A(X)(e),\]
thus $H(X)\cong H_A(X)$ naturally. Since $A(e)=T(\CE(-,e))\in\CC$
for all $e\in\CE$, the closure of $\CC$ under colimits and the
formula
\[T_A(Y)=\colim_{(e,y)\in\CE\downarrow Y}A(e),\] valid for all
$Y\in[\CE\opp,\Set]$, show that $\Ima T_A\subseteq\CC$. Therefore
the assignment $Y\mapsto T_A(Y)$ defines a functor $\CC'\to\CC$,
which is naturally isomorphic to $T$, as right adjoints of $H$.

Further $\Ima T_A\subseteq\CC$ implies $\Ima H_A\subseteq\CC'$,
since for all $X\in\CA$, we have  \[(H_A\circ T_A\circ
H_A)(X)\cong H((T_A\circ H_A)(X))\in\CC',\] $H_A(X)$ is a retract
of $(H_A\circ T_A\circ H_A)(X)$ and $\CC'$ is closed under
retracts.

Now, the fully faithfulness of $T$ is equivalent to the fact that
$H\circ T\cong\id_{\CC'}$ naturally, thus $\Ima H_A\in\CC$ implies
\[(H_A\circ T_A\circ H_A)(X)\cong(H\circ T)(H_A(X))\cong
H_A(X),\] what means that $\delta_X$ is an $H_A$-equivalence.
\end{proof}

\begin{theorem}\label{repres} Let $H:\CC\rightleftarrows\CC':T$ be
mutually inverse equivalences of categories, and denote
$A:\CE\to\CA$ the functor given by $A(e)=T(\CE(-,e))$. If
$\CE(-,e)\in\CC'$ for all $e\in\CE$ then $H$ is naturally
isomorphic to the restriction of $H_A$, $T$ is naturally
isomorphic to the restriction of $T_A$, $\CC=\CC_A$ and
$\CC'=\CC^A$.
\end{theorem}

\begin{proof}
The conclusions concerning $H$ and $T$ follow by Lemma
\ref{repadj}. For the rest, we have for all $X\in\CC$:
\[X\cong(T\circ H)(X)\cong(T_A\circ H_A)(X)\in\Ima
T_A\subseteq\CC_A\] as we have seen in Lemma \ref{inclusions}.
Thus $\CC\subseteq\CC_A$, and dually $\CC'\subseteq\CC^A$.

The functor $\CA\to\CC$ given by $X\mapsto(T_A\circ H_A)(X)$ is a
right adjoint of the inclusion functor $\CC\to\CA$. Indeed, for
all $X'\in\CC$ and all $X\in\CA$, we obtain $X'\cong(T\circ
H)(X')\cong(T_A\circ H_A)(X')$ and $H_A(X)\cong(H_A\circ T_A\circ
H_A)(X)$ since the counit $\delta_X$ of adjunction is an
$H_A$-equivalence, as we observed in Lemma \ref{repadj}. Now the
natural isomorphisms
\begin{align*} \CA(X',(T_A\circ H_A)(X))&\cong\CA((T_A\circ
H_A)(X'),(T_A\circ H_A)(X))\\ &\cong[H_A(X'),(H_A\circ T_A\circ
H_A)(X)]\cong[H_A(X'),H_A(X)]\\ &\cong\CA((T_A\circ
H_A)(X'),X)\cong\CA(X',X)\end{align*} prove our claim. Using again
the fact that $\delta_X:(T_A\circ H_A)(X)\to X$ is an
$H_A$-equivalence, Lemma \ref{smallest} tells us that
$\CC_A\subseteq\CC$. The functor $[\CE\opp,\Set]\to\CC'$ given by
$Y\mapsto(H_A\circ T_A)(Y)$ is also well defined. In a dual manner
we show that it is a left adjoint of the inclusion
$\CC'\to[\CE\opp,\Set]$, and follows $\CC^A\subseteq\CC'$.
\end{proof}

The equivalences $H:\CC\rightleftarrows\CC':T$ are called {\em
represented} by $A:\CE\to\CA$ provided that $H\cong H_A$ and
$T\cong T_A$ as in the Theorem \ref{repres}.

In the work \cite{MO} of Menini and Orsatti (see also \cite{CF}),
it is given an additive version of Theorem \ref{repres}. There,
our category $\CE$ is preadditive with a single object (that means
it is a ring), $A$ is an object in $\CA$ with endomorphism ring
$\CE$ (therefore $A:\CE\to\CA$ is a fully faithful functor), and
$[\CE\opp,\Set]$ is replaced with $\Md(\CE)$.

\section{The existence of cellular covers}\label{ecc}

In this Section consider as in the previous one a cocomplete
category $\CA$, a functor $A:\CE\to\CA$, where $\CE$ is a small
category and construct its left Kan extension
$T_A:[\CE\opp,\Set]\to\CA$ along the Yoneda embedding
$\CE\to[\CE\opp,\Set]$ which has the right adjoint
$H_A:\CA\to[\CE\opp,\Set]$. In addition suppose that $A$ is fully
faithful. Note that, this additional assumption means that the
category $\CE$ may be identified with a (small) subcategory of
$\CA$ and $A$ with the inclusion functor. For example, if $\CE$
has a single object, then $A$ may be identified with an object of
$\CA$.

\begin{lemma}\label{smallg}
If $\CA$ is a cocomplete, balanced category with epimorphic images
and $A:\CE\to\CA$ is fully faithful, then it holds:
\begin{itemize}
\item[{\rm a)}] $A(e)\in\CS_A$ for all $e\in\CE$. \item[{\rm b)}]
An object $X\in\CA$ is $A$-generated exactly if there is an
epimorphism $A'\to X$ with $A'$ a coproduct of objects of the form
$A(e)$ with $e\in\CE$.
\end{itemize}
\end{lemma}

\begin{proof} a) Since $A$ is fully faithful, we have the natural
isomorphisms: \[(T_A\circ H_A)(A(e))=T_A(\CA(A(-),A(e))\cong
T_A(\CE(-,e))\cong A(e),\] for every $e\in\CE$.

b) Let $A'=\coprod A(e_i)\in\CA$ be a coproduct of objects of the
form $A(e)$. By the result in a) we deduce \[A'=\coprod
A(e_i)\cong\coprod(T_A\circ H_A)(A(e_i))\cong T_A(\coprod
H_A(A(e_i)))\in\Ima T_A,\] so $A'\in\CG_A$, since $\Ima
T_A\subseteq \CG_A$, inclusion established in Lemma
\ref{gclosepi}. If $X\in\CA$ such that there is an epimorphism
$A'\to X$, then $X\in\CG_A$, again by Lemma \ref{gclosepi}.
Conversely, for every $X\in\CA$, the object $H_A(X)$ of
$[\CE\opp,\Set]$ may be written as
\[H_A(X)\cong\colim_{(e,x)\in\CE\downarrow H_A(X)}\CE(-,e)=\colim_{(e,x)\in A\downarrow X}\CE(-,e),\]
where the comma category $A\downarrow X$ has as objects pairs of
the form $(e,x)$ with $e\in\CE$ and $x\in\CA(A(e),X)$. Thus
\[(T_A\circ H_A)(X)\cong \colim_{(e,x)\in A\downarrow
X}T_A(\CE(-,e))\cong\colim_{(e,x)\in A\downarrow X}A(e),\] so
there is an epimorphism from $\coprod_{(e,x)\in A\downarrow
X}A(e)$ to $(T_A\circ H_A)(X)$. Further the morphism
$\delta_X:(T_A\circ H_A)(X)\to X$ is an epimorphism too, for
$X\in\CG_A$. Composing them we obtain the desired epimorphism.
\end{proof}

Recall that a category $\CA$ is called {\em well (co)powered} if
for every object the class of subobjects (respectively quotient
objects) is actually a set.

\begin{theorem}\label{ctoa} If $\CA$ is a cocomplete,
well copowered, balanced category with epimorphic images, and
$A:\CE\to\CA$ is a fully faithful functor, then the inclusion
functor $\CC_A\to\CA$ has a right adjoint, or equivalently, every
object in $\CA$ has a left $\CC_A$-approximation.
\end{theorem}

\begin{proof} Combining Lemma \ref{smallg} and
Corollary \ref{csubsetg}, we deduce that every object in $\CC$ is
a quotient object of a direct sum of objects of the form $A(e)$,
so $\{A(e)\mid e\in\CE\}\subseteq\CC_A$ is a generating set for
$\CC_A$, by \cite[Chapter II, Proposition 15.2]{BM}. The closure
of $\CC_A$ under colimits implies that the inclusion functor
$\CC_A\to\CA$ preserves colimits, and the category $\CC_A$
inherits from $\CA$ the property to be well copowered. Thus the
conclusion follows by Freyd's Special Adjoint Functor Theorem (see
\cite[Chapter V, Corollary 3.2]{BM}).
\end{proof}

If $\CE$ has a single object and $A$ ia a fully faithful functor
(i.e. $A$ is an object of $\CA$), then $A$-colocal objects are
sometimes called {\em$A$-cellular}, and an $H_A$-equivalence is
called then simply an {\em$A$-equivalence}. Our Theorem \ref{ctoa}
shows that, under reasonable hypotheses (that means $\CA$ is a
cocomplete, well copowered, balanced category with epimorphic
images), every object $X$ has an {\em$A$-cellular approximation},
what means an $A$-equivalence $C\to X$ with $C$ being
$A$-cellular. Hence it is generalized in this way \cite[Section
2.C]{FGS}, where is constructed an $A$-cellular approximation for
every group.

The same proof that given \cite[Lemma 2.6]{FGS} for the case of
the category of groups works for the following consequence of the
existence of an $A$-cellular approximation for every object
$X\in\CA$:

\begin{corollary}\label{initer}
Let $A:\CE\to\CA$ be a fully faithful functor, where $\CA$ is a
cocomplete, well copowered, balanced category with epimorphic
images and $\CE$ is small category. The following are equivalent
for a morphism $\alpha:C\to X$ in $\CA^\to$: \begin{itemize}
\item[{\rm (i)}] $\alpha$ is a left $\CC_A$-approximation of $X$.
\item[{\rm (ii)}] $\alpha$ is an $H_A$-equivalence and it is
initial among all $H_A$-equivalences ending in $X$ \item[{\rm
(iii)}] $C\in\CC_A$ and $\alpha$ is terminal among all morphisms
from an $A$-colocal object to $X$.\end{itemize}
\end{corollary}

Consequently we may use the following more of less tautological
formulas for determining the left $\CC_A$-approximation of an
object (see \cite[Sections 7.1 and 7.2]{DL}):

\begin{corollary}\label{bousfield}
Let $A:\CE\to\CA$ be a fully faithful functor, where $\CA$ is a
cocomplete, well copowered, balanced category with epimorphic
images and $\CE$ is small category. If $\alpha:C\to X$ is the left
$\CC_A$-approximation $X$, then it holds:

\begin{itemize} \item[{\rm a)}]
$C=\displaystyle{\lim_{X'\to X}}X'$, where $X'\to X$ runs over all
$H_A$-equivalences. \item[{\rm b)}] $C=\displaystyle{\colim_{X'\to
X}}X'$, where $X'$ runs over all $A$-colocal objects.
\end{itemize}
\end{corollary}

\section{$*$-acts over monoids}\label{nonadd}

We see a monoid $M$ as a category with one object whose
endomorphism set is $M$. Thus we consider the category
$[M\opp,\Set]$ of all contravariant functors from this category to
the category of sets, and we call it the category of {\em(right)
acts} over $M$, or simply {\em$M$-acts}. Clearly an $M$-act is a
set $X$ together with a an action $X\times M\to X$, $(x,m)\mapsto
xm$ such that $(xm)m'=x(mm')$ and $x1=x$ for all $x\in X$ and all
$m,m'\in M$. Left acts are covariant functors $M\to\Set$, that is
sets $X$ together with an action $M\times X\to X$, satisfying the
corresponding axioms. For the general theory of acts over monoids
and undefined notions concerning this subject we refer to
\cite{KKM}. We should mention here that in contrast with
\cite{KKM} we allow the empty act to be an object in our category
of acts, for the sake of (co)completness. Note that the category
of $M$-acts is balanced and has epimorphic images, by
\cite[Proposition 1.6.15 and Theorem 1.4.21]{KKM}.

Fix a monoid $M$ and an object $A\in[M\opp,\Set]$. In order to use
the results of the preceding Sections, we identify $A$ with a
fully faithful functor $E\to[M\opp,\Set]$ where $E$ is the
endomorphism monoid of $A$. Thus $A$ is canonically a $E-M$-biact
(see \cite[Definition 1.4.24]{KKM}), so we obtain two functors
\[H_A:[M\opp,\Set]\to[E\opp,\Set],\ H_A(X)=[A,X]\] and
\[T_A:[E\opp,\Set]\to[M\opp,\Set],\ T_A(Y)=Y\otimes_EA\]
the second one being the left adjoint of the first (see
\cite[Definition 2.5.1 and Proposition 2.5.19]{KKM}). Clearly
these functors agree with the functors defined at the beginning of
the Section \ref{repeq}.

We say that $A$ is a {\em (weak) $*$-act} if the above adjoint
pair induces mutually inverse equivalences
$H_A:\CG_A\rightleftarrows\CG^A:T_A$ (respective
$H_A:\CC_A\rightleftarrows\CG^A:T_A$). Note that our definitions
for subcategories $\CG_A$ and $\CG^A$ agree with the
characterizations of all $A$-generated respectively
$A^*$-cogenerated modules given in \cite[Lemma 2.1.2]{CFB}. As we
may see from Proposition \ref{wstarob}, our subcategory $\CC_A$
seems to be the non--additive counterpart of the subcategory of
all $A$-presented modules (compare with \cite[Proposition
2.2.4]{CFB}).

In what follows, we need more definitions relative to an $M$-act
$A$. First $A$ is called {\em decomposable} if there exists two
non empty subacts $B,C\subseteq A$ such that $A=B\cup C$ and
$B\cap C=\emptyset$ (see \cite[Definition 1.5.7]{KKM}). In this
case $A=B\sqcup C$, since coproducts in the category of acts is
the disjoint union, by \cite[Proposition 2.1.8]{KKM}. If $A$ is
not decomposable, then it is called {\em indecomposable}. Second,
$A$ is say to be {\em weak self--projective} provided that
$(H_A\circ T_A)(g)$ is an epimorphism whenever $g:U\to Y$ is an
epimorphism in $[E\opp,\Set]$ with $U\in\CS^A$. More explicitly,
if $g:U\to Y$ is an epimorphism in $[E\opp,\Set]$, then  $T_A(g)$
is an epimorphism in $[M\opp,\Set]$ and our definition requires
that $A$ is projective relative to such epimorphisms for which
$U\in\CS^A$. Third $A$ is called {\em (self--)small} provided that
the functor $H_A$ preserves coproducts (of copies of $A$).

\begin{lemma}\label{indec} With  the notations above, the following are equivalent:
\begin{itemize}
\item[{\rm (i)}] $A$ is small. \item[{\rm (ii)}] $A$ is
self--small. \item[{\rm (iii)}] $E^{(I)}$ is $\eta$-reflexive for
any set $I$, where $E^{(I)}$ denotes the coproduct indexed over
$I$ of copies of $E$. \item[{\rm (iv)}] $E\sqcup E$ is
$\eta$-reflexive. \item[{\rm (v)}] $A$ is indecomposable.
\end{itemize}
\end{lemma}

\begin{proof} (i)$\Rightarrow$(ii) is obvious.

(ii)$\Rightarrow$(iii). If $H_A$ commutes with coproducts of
copies of $A$ then
\begin{align*} E^{(I)}&=\coprod_I[A,A]\cong\left[A,\coprod_IA\right]
\cong\left[A,\coprod_I(E\otimes_EA)\right]\\
&\cong\left[A,\left(E^{(I)}\right)\otimes_EA\right]\cong(H_A\circ
T_A)\left(E^{(I)}\right).\end{align*}

(iii)$\Rightarrow$(iv) is obvious.

(iii)$\Rightarrow$(iv). If $A$ is decomposable, that is $A=B\sqcup
C$ with $B\neq\emptyset$ and $C\neq\emptyset$, then let $i_B:B\to
A$ and $i_C:C\to A$ the canonical injections of this coproduct.
Denote also by $j_1,j_2:A\to A\sqcup A$ the corresponding
canonical injections. The homomorphisms of $M$-acts $j_1i_B:B\to
A\sqcup A$ and $j_2i_C:C\to A\sqcup A$ induce a unique
homomorphism $f:A=B\sqcup C\to A\sqcup A$. Obviously
$f\in(H_A\circ T_A)(E\sqcup E)$ but
$f\notin[A,A]\sqcup[A,A]=E\sqcup E$.

(iv)$\Rightarrow$(i) is \cite[Lemma 1.5.37]{KKM}.

\end{proof}

\begin{proposition}\label{wstaract}
The following statements hold:
\begin{itemize}
 \item[{\rm a)}] If $A$ is a weak $*$-act then $A$ is weak self--projective.
 \item[{\rm b)}] If $A$ is weak self--projective and indecomposable, then $A$ is a weak $*$-act.
\end{itemize}
\end{proposition}

\begin{proof} a) Let $A$ be a weak $*$-act and let $g:U\to Y$ be an epimorphism
in $[E\opp,\Set]$ with $U\in\CS^A$. We know by Proposition
\ref{wstarob} that $\eta_Y$ is epic, and by the naturalness of
$\eta$ that $(H_A\circ T_A)(g)\eta_U=\eta_Yg$. Since $\eta_U$ is
an isomorphism and $\eta_Yg$ is an epimorphism we deduce that
$(H_A\circ T_A)(g)$ is an epimorphism too.

b) As we have already noticed $H_A$ preserves coproducts, provided
that $A$ is indecomposable. Thus $\CS^A$ is closed under arbitrary
coproducts in the category of $E$-acts. For a fixed
$Y\in[E\opp,\Set]$ there is an epimorphism $g:E^{(I)}\to Y$. How
$E$ is $\eta$-reflexive the same is also true for $E^{(I)}$. But
$(H_A\circ T_A)(g)$ is an epimorphism, since $A$ is weak
self--projective. From the equality $(H_A\circ
T_A)(g)\eta_{E^{(I)}}=\eta_Yg$ follows that $\eta_Y$ is an
epimorphism too. The conclusion follows by Proposition
\ref{wstarob}.
\end{proof}

\begin{theorem}\label{staract}
The following statements hold:
\begin{itemize}
 \item[{\rm a)}] If $A$ is a $*$-act then $A$ is weak self--projective and $\CC_A=\CG_A$.
 \item[{\rm b)}] If $A$ is indecomposable, weak self--projective and $\CC_A=\CG_A$, then $A$ is a $*$-act.
\end{itemize}
\end{theorem}

\begin{proof} The both implications follow at once from Proposition \ref{wstaract}. \end{proof}

\begin{remark}
Propositions \ref{wstarob} and \ref{wstaract} and Theorems
\ref{starob} and \ref{staract} provide a non additive version of
\cite[Proposition 2.2.4]{CFB} respectively
\cite[Theorem2.3.8]{CFB}. In contrast with the case of modules,
where the functors are additive, for acts it is not clear that a
weak star object must me indecomposable (the non additive version
of self--smallness as we may seen from Lemma \ref{indec}). The
main obstacle for deducing this implication in the new setting
comes from the fact that non additive functors do not have to
preserves finite coproducts.
\end{remark}

Using the characterization of so called tilting modules given in
\cite[Theorem 2.4.5]{CFB}, we may define a {\em tilting} $M$-act
to be a $*$-act $A$ such that the injective envelope of $M$
belongs to $\CG_A$. Note that injective envelopes exist in
$[M\opp,\Set]$ by \cite[Corollary 3.1.23]{KKM}. As a subject for a
future research we may ask ourselves which from the many beautiful
results which are known for tilting modules do have correspondents
for acts.

Our next aim is to infer from our results the Morita--type
characterization of an equivalence between categories of acts (see
\cite[Section 5.3]{KKM}). In order to perform it we need a couple
of lemmas.

\begin{lemma}\label{genact}
If the $M$-act $A$ is a generator in $[M\opp,\Set]$ then
$\CC_A=\CG_A=[M\opp,\Set]$.
\end{lemma}

\begin{proof}
For a generator $A$ of $[M\opp,\Set]$ the equality
$\CG_A=[M\opp,\Set]$ follows by Lemma \ref{smallg}. Moreover $M$
is a retract of $A$ by \cite[Theorem 2.3.16]{KKM}, therefore
$M\in\CC_A$, since $\CC_A$ is closed under retracts. Thus a
morphism $\epsilon:U\to V$ in $[M\opp,\Set]$ is an $A$-equivalence
if and only if it is an isomorphism, therefore
$\CC_A=[M\opp,\Set]$.
\end{proof}

Recall that the left $E$-act $A$ is said to be {\em pull back
flat} if the functor $T_A=(-\otimes_EA)$ commutes with pull backs
(see \cite[Definition 3.9.1]{KKM}).

\begin{lemma}\label{wprojact}
If the right $M$-act $A$ is indecomposable, weak self projective
and the left $E$-act $A$ is pull back flat, then
$\CG^A=[E\opp,\Set]$.
\end{lemma}

\begin{proof} First observe that $A$ is a weak $*$-act by Proposition
\ref{wstaract}. Hence $\CG^A=\CC^A=\CS^A$, and this subcategory
has to be closed under subacts and limits. Moreover $E^{(I)}$ is
$\eta$-reflexive for any set $I$ according to Lemma \ref{indec}.
For a fixed $Y\in[E\opp,\Set]$ there is an epimorphism
$g:E^{(I)}\to Y$. Take the kernel pair of $g$, that is construct
the pull back
\[\diagram K\rto^{k_1}\dto_{k_2}&E^{(I)}\dto^{g}\\
            E^{(I)}\rto_{g}&Y\enddiagram.\]
The functors $T_A$ and $H_A$ preserve pull backs, the first one by
hypothesis and the second one automatically. Moreover $K$ is a
subact of $E^{(I)}\times E^{(I)}$ and the closure properties of
$\CS^A$ imply $K\cong(T_A\circ H_A)(K)$. Applying the functor
$H_A\circ T_A$ to the above diagram and having in the mind the
previous observations we obtain a
pull back diagram \[\diagram K\rrto^{k_1}\dto_{k_2}&&E^{(I)}=(H_A\circ T_A)\left(E^{(I)}\right)\dto^{(H_A\circ T_A)(g)}\\
            E^{(I)}=(H_A\circ T_A)\left(E^{(I)}\right)\rrto_{(H_A\circ T_A)(g)}&&(H_A\circ T_A)(Y)\enddiagram.\]
Note that $(H_A\circ T_A)(g)$ is an epimorphism by hypothesis.
Then we know by \cite[Theorem 2.2.44]{KKM} that both $g$ and
$(H_A\circ T_A)(g)$ are coequalizers for the pair $(k_1,k_2)$.
Thus we deduce $Y\cong(H_A\circ T_A)(Y)$ canonically, so
$Y\in\CS^A$. Thus $\CG^A=\CS^A=[E\opp,\Set]$.
\end{proof}

Now we are in position to prove the desired Morita--type result:

\begin{theorem}\label{Morita}
Let $M$ and $E$ be two monoids. Then the categories $[M\opp\Set]$
and $[E\opp,\Set]$ are equivalent via the mutually inverse
equivalence functors $H$ and $T$ if and only if there is a cyclic,
projective generator $A$ of $[M\opp,\Set]$ such that $E$ is the
endomorphism monoid of $A$, case in which $H=H_A$ and $T=T_A$.
\end{theorem}

\begin{proof} First note that a projective act is indecomposable if and
only if it is cyclic in virtue of \cite[Propositions 1.5.8 and
3.17.7]{KKM}.

If $H:[M\opp,\Set]\rightleftarrows[E\opp,\Set]:T$ are mutually
inverse equivalences, then $H\cong H_A$ and $T\cong T_A$, where
$A=T(E)$ according to Theorem \ref{repres}. Moreover the
endomorphism monoid of $A$ is $E$, and $A$ has to be projective,
indecomposable and generator together with $E$.

Conversely if $A$ is indecomposable and projective in
$[M\opp,\Set]$ then it is a weak $*$-act by Proposition
\ref{wstaract}.  Since $A$ is in addition a generator, Lemma
\ref{genact} tell us that $A$ is a $*$-act and
$\CC_A=\CG_A=[M\opp,\Set]$ and Theorem \ref{staract} implies that
$A$ is a $*$-act. Finally the left $E$-act $A$ is projective by
\cite[Corollary 3.18.17]{KKM}, so it is strongly flat by
\cite[Proposition 3.15.5]{KKM}, that means $T_A$ commutes both
with pull backs and equalizers. Thus $\CG^A=[E\opp,\Set]$,
according to Lemma \ref{wprojact}.
\end{proof}

\end{document}